\documentclass[runningheads]{amsart}
\usepackage{amsmath}
\usepackage{amsfonts}
\usepackage{epsfig}
\usepackage{graphics}

\def\meas{{\rm meas}\, }
\def\diam{{\rm diam}\, }

\begin{document}
\newtheorem{theorem}{Theorem}[section]
\newtheorem{lemma}[theorem]{Lemma}
\theoremstyle{definition}
\newtheorem{definition}[theorem]{Definition}
\newtheorem{exercise}[theorem]{Exercise}
\newtheorem{example}[theorem]{Example}
\newtheorem{remark}[theorem]{Remark}
\title{On generalizations of the Synge-K\v r\'{\i}\v zek maximum angle condition for $d$-simplices}


\author{ Ali Khademi \and Sergey Korotov \and Jon Eivind Vatne}


\address{
Department of Computing, Mathematics and Physics, Western Norway University of Applied Sciences,
Post Box 7030, N-5020 Bergen, Norway, {\tt ali.khademi@hvl.no, a.khademi.math@gmail.com, sergey.korotov@hvl.no, jon.eivind.vatne@hvl.no}
}
\maketitle

\begin{abstract}
  In this note we present a generalization of the maximum angle condition,
  proposed by J. L. Synge in 1957 and M. K\v r\'{\i}\v zek in 1992 for trianglular
  and tetrahedral elements, respectively, for the case of higher-dimensional
  simplicial finite elements. Its relations to the other  angle-type
  conditions commonly used in finite element methods are analysed.
\end{abstract}

\section{Introduction}
\label{chapter-intro}

Let $\mathcal F = \{\mathcal T_h\}_{h \to 0}$ be a family of conforming (face-to-face)
triangulations $\mathcal T_h$ of a bounded polygonal domain.
In 1957, see \cite{Syn}, Synge proved that linear triangular
finite elements yield the optimal interpolation order in the $C$-norm provided the
following {\it maximum angle condition} is satisfied: there exists a constant $\gamma_0 < \pi$
such that for any triangulation $\mathcal T_h \in \mathcal F$ and any triangle $T \in \mathcal T_h$
the upper bound
\begin{equation}
  \gamma_T \le \gamma_0,
  \label{eq-upper-bound}
\end{equation}
holds, where $\gamma_T$ is the maximum angle of $T$. Later, Babu\v ska and Aziz \cite{BabAzi},
Barnhill and Gregory \cite{BarGre}, and Jamet \cite{Jam} independently derived the optimal interpolation
order in the energy norm of finite element approximations under the condition (\ref{eq-upper-bound}),
see also \cite{Kri-1991} in this respect.

In 1992, the Synge-condition (\ref{eq-upper-bound}) was generalized by K\v r\'{\i}\v zek \cite{Kri-1992}
to tetrahedral elements as follows: there exists a constant $\gamma_0 < \pi$
such that for any face-to-face tetrahedralization $\mathcal T_h \in \mathcal F$ and any
tetrahedron $T \in \mathcal T_h$ one has
\begin{equation}
  \gamma_{\rm D} \le \gamma_0 \quad \text{and}   \quad   \gamma_{\rm F} \le \gamma_0,
  \label{eq-upper-bound-3D}
\end{equation}
where $\gamma_{\rm D}$ is the maximum dihedral angles between faces of $T$ and
$\gamma_{\rm F}$ is the maximum angle in all four triangular faces of $T$. The
optimal interpolation estimates were obtained in \cite{Kri-1992} for various norms
under the condition (\ref{eq-upper-bound-3D}),
thus allowing the usage of many degenerating (skinny or flat)
tetrahedra unavoidably appearing during mesh generation and adaptivity processes
in various real-life applications \cite{CheDeyEdeFacTen,Ede}.

Recently, some higher-dimensional generalization of conditions (\ref{eq-upper-bound})
(and its relation to condition (\ref{eq-upper-bound-3D})) was proposed and analysed
in \cite{HanKorKri-2017} (see also \cite{HanKorKri-2018}).
However, this generalization is not of the
form of an upper estimate for (all or some) angles of the simplices generated
(cf. Definition \ref{def-semiregular-family}).


\section{Angle conditions in higher dimensions}
\label{chapter-denotation}

Recall that a $d$-simplex $S$ in ${\bf R}^d, d \in \{1, 2, 3, \dots \}$,  is the convex hull of $d+1$
vertices $A_0, A_1, \dots, A_{d}$ that do not belong to the same $(d-1)$-dimensional
hyperplane, i.e.,
$$
S =  {\rm conv}  \{A_0, A_1, \dots , A_{d}\}.
$$
Let
$$
F_i =  {\rm conv} \{A_0, \dots , A_{i-1}, A_{i+1}, \dots ,A_{d}\}
$$
be the facet of $S$ opposite to the vertex $A_i$ for $i \in \{0, \dots, d\}$.

For $d \ge 2$ the dihedral angle $\beta_{ij}$ between two facets $F_i$ and $F_j$ of $S$
is defined by means of the inner product of their outward unit normals $n_i$ and $n_j$
\begin{equation*}
\cos \beta_{ij} = - n_i  \cdot n_j.
\label{dihedral-angle}
\end{equation*}

In 1978, Eriksson introduced a generalization of the sine function to an arbitrary
$d$-dimensional spatial angle, see \cite[p.~74]{Eri}.

\begin{definition}
  Let $\hat A_i$ be the angle at the vertex $A_i$ of the simplex $S$. Then  $d$-sine of the angle
  $\hat A_i$ for $d>1$ is given by
\begin{equation}
\sin_d (\hat A_{i}  | A_0 A_1 \dots A_{d})
= \frac{ d^{d-1} \, (\meas_d S)^{d-1} }{(d-1) ! \, {\bf \Pi}_{j = 0, j \neq i}^{d} \meas_{d-1} F_j }.
\label{d-sine}
\end{equation}
\end{definition}

\begin{remark}
  The $d$-sine is really a generalization of the classical sine function. In order to
  see that consider an arbitrary triangle $A_{0} A_{1} A_{2}$. Let $\hat A_{0}$
be its angle at the vertex $A_{0}$. Then, obviously,
\begin{equation*}
\meas_2 (A_0 A_1 A_2) = {1\over 2} |A_0 A_1| |A_0 A_2| \sin \hat A_0.
\label{2d-area-triangle}
\end{equation*}
Comparing this relation with (\ref{d-sine}) for $d=2$, we find that
\begin{equation*}
  \sin \hat A_0 = \sin_2 (\hat A_0 | A_0 A_1 A_2).
\label{eq-compare}
\end{equation*}
\end{remark}

\begin{definition}
  A family $\mathcal F = \{\mathcal T_h \}_{h \to 0}$  of face-to-face partitions of a polytope
   $\overline \Omega \subset {\bf R}^d$ into
$d$-simplices is said to satisfy the {\it generalized minimum angle condition} if
there exists a constant $C > 0$ such that for any $\mathcal T_h \in \mathcal F$ and any
$S = {\rm conv} \{ A_0, \dots , A_{d} \} \in \mathcal T_h$
one has
\begin{equation}
\forall \ i \in \{0, 1, \dots , d \}
\qquad
\sin_d ( \hat A_i  | A_{0} A_{1} \dots A_{d})  \ge C > 0.
\label{conditions-4-any-dimension}
\end{equation}
\label{definition-min-angle-cond-4-any-dimension}
\end{definition}

This condition is investigated in the paper \cite{BraKorKri-Zlamal}. It generalizes the well-known
Zl\'amal minimum angle condition for triangles (see \cite{Cia,Zen,Zla}), which is stronger
than~(\ref{eq-upper-bound}).

\begin{definition}
  A family $\mathcal F  = \{\mathcal T_h \}_{h \to 0}$ of face-to-face partitions of a polytope
  $\overline \Omega \subset {\bf R}^d$ into
$d$-simplices is said to satisfy the {\it generalized maximum angle condition} if
there exists a constant $C > 0$ such that for any
$\mathcal T_h \in \mathcal F$ and any $S = {\rm conv} \{ A_0, \dots , A_{d} \} \in \mathcal T_h$
one can always choose $d$ edges of $S$, which, when considered as vectors, constitute
a (higher-dimensional) angle whose $d$-sine is bounded from below by the constant $C$.
\label{def-semiregular-family}
\end{definition}

\begin{remark}
  The generalized maximum angle condition is really weaker than
  the generalized minimum angle condition as it accepts e.g. degenerating
  path-simplices \cite{BraKorKriSol}, which obviously violate
  Definition \ref{definition-min-angle-cond-4-any-dimension}.
\end{remark}

The main result on the interpolation estimate is given in the following theorem.

\begin{theorem}
  Let $\mathcal F$ be a family of face-to-face
  partitions of a polytope $\overline \Omega \subset {\bf R}^d$
  into $d$-simplices satisfying
  the generalized maximum angle condition from Definition \ref{def-semiregular-family}. Then
there exists a constant $C > 0$ such that for any $\mathcal T_h \in \mathcal F$ and any
$S \in \mathcal T_h$ we have
\begin{equation*}
\| v - \pi_S v\|_{1, \infty} \le C h_S |v|_{2, \infty} \qquad \forall v \in \mathcal C^2(S),
\label{main-theorem}
\end{equation*}
where $\pi_S$ is the standard Lagrange linear interpolant and $h_S = \diam S$.
\end{theorem}

For the proof see \cite{HanKorKri-2017}.

\begin{definition}
    A family $\mathcal F  = \{\mathcal T_h \}_{h \to 0}$ of face-to-face partitions of a polytope
  $\overline \Omega \subset {\bf R}^d$ into
$d$-simplices is said to satisfy the {\it $d$-dimensional  maximum angle condition} if
    there exists a constant $\gamma_0<\pi$ such that for $\mathcal T_h \in \mathcal F$
    and any simplex $S \in \mathcal T_h$  and any subsimplex $S'\subseteq S$ with vertex
    set contained in the vertex set of $S$, the maximum dihedral angle in $S'$ is less
    than or equal to $\gamma_0$.
\label{def-krizek-general}
\end{definition}

\begin{remark}
  It is worth to mention that the  maximum angle condition is only sufficient
  to provide the convergence of the finite element approximations as shown in
  \cite{HanKorKri}.
\end{remark}

\section{Main results}

In this section we present the main results of the work.

\begin{lemma}
For a $d$-simplex we observe that
 \begin{equation}
\sin_d (\hat A_{i}  | A_0 A_1 \dots A_{d})
= \sin_{d-1} (\hat A_{i}  | A_0 A_1 \dots A_{d-1})\prod_{j=0, j\neq i}^{d-1}\sin(\beta_j),
\label{d-sine-product}
\end{equation}
 where $\beta_j$ is the dihedral angle between the facet opposite to $A_j$ and
 the facet opposite to  $A_d$.
\label{lemma-sines-tetrahedron-ABCD}
\end{lemma}

For the proof see \cite[p. 74--76]{Eri}.

\begin{remark}
  The immediate consequence of Lemma \ref{lemma-sines-tetrahedron-ABCD}
  is that $\sin_d$ is always less than or equal to one.
\end{remark}
As usual, we denote by $S^{d-1}$ the unit sphere in $\mathbb{R}^d$, and by $(S^{d-1})^N$ the $N$-fold cartesian product, i.e. the space of $N$ unit vectors in $\mathbb{R}^d$.
\begin{lemma}[Properties of $\sin_d$]
  We can define $\sin_d$ as a function on the space of $d$ unit vectors
  $\vec{t}_1,\dots,\vec{t}_d$
  in $\mathbb{R}^d$, $\left(S^{d-1}\right)^d$, with the following properties:
\begin{itemize}
\item[a)] On the open, dense subset of $d$ vectors spanning $\mathbb{R}^d$
  (linearly independent), $\sin_d(\vec{t}_1,\dots,\vec{t}_d)$ is equal to
  the expression from \eqref{d-sine} with $\hat{A}_i$ equal to the origin
  and the other points from $\{A_0,\dots,A_d\}$ equal to the endpoints of
  the vectors.
\item[b)] On the closed subspace of $d$ vectors not spanning $\mathbb{R}^d$
  (linearly dependent), $\sin_d(\vec{t}_1,\dots,\vec{t}_d)=0$.
\item[c)] $\sin_d$ is continuous.
\end{itemize}
\label{lemma-cont}
\end{lemma}
\begin{proof}
All these properties are implicitly found in Eriksson's work \cite{Eri}.

Part a): On page 72 of \cite{Eri}, he notes that the definition of $\sin_d$
does not change if one of the vectors is multiplied by a nonzero constant.
Thus we can normalize all the vectors and use only unit vectors.

Part b): We take this as the definition of $\sin_d$ for linearly dependent
vectors.

Part c): On the two sets in a) and in b) considered separately, $\sin_d$ is
obviously continuous. We must check that when a tuple of linearly independent
unit vectors $(\vec{t}_1,\dots,\vec{t}_d)$ approaches a linearly dependent
limit, $\sin_d$ approaches zero.
This is clear for $d=2$ and can be proved by induction for any $d$ using the product formula from Lemma~\ref{lemma-sines-tetrahedron-ABCD}:
If any subset of the vectors becomes close to a linearly dependent set, we can apply the product formula so that the evaluation of $\sin_{d-1}$ involves those vectors, so the product tends to zero by the induction hypothesis.
Otherwise, to get a degeneration, a dihedral angle must tend to zero or $\pi$,
so that the product tends to zero by the continuity of the ordinary (two-dimensional) sine.
\end{proof}
\begin{remark}
It is necessary to restrict the domain of definition for $\sin_d$ in order to have continuity.
Otherwise we would have the following problem:
If one of the vectors tend to zero, the limit would include a zero vector and thus be a linearly dependent set.
By part b), $\sin_d$ should be zero.
But since multiplying an edge by any nonzero constant leaves $\sin_d$ unchanged, this would violate continuity.
\end{remark}

Let $S_n= \operatorname{conv}\{A^n_0,A^n_1,\dots,A^n_d\}$, $n=1, 2,\dots$, be any infinite sequence
of simplices, and consider the vectors
\begin{equation*}
\vec{t}_{X_nY_n} = \frac{X_nY_n}{|X_nY_n|}\quad \text{ for any pair } \{X_n,Y_n\}\subset\{A^n_0,A^n_1,\dots,A^n_d\},
\end{equation*}
where the symbol $|\cdot|$ denotes the length of the vector.

\begin{lemma}
There is a subsequence $\{S_{n'}\}\subset\{S_n\}$ such that all the sequences $\{\vec{t}_{X_nY_n}\}$ converge.
\label{lemma-subsequence}
\end{lemma}
\begin{proof}
  The sequence of tuples $\{\big(\vec{t}_{X_nY_n}\big)| \text{ for all pairs }\{X_n,Y_n\} \subset\{A^n_0,\dots,A^n_d\}\}$
  is an infinite subset of the space $(S^{d-1})^N$, where $N=\binom{d+1}{2}$ is the number of pairs.
Since this space is compact, the sequence has at least one limit point.
Let $\{S_{n'}\}$ be a subsequence converging to such a limit point.
\end{proof}

\begin{theorem}
  Let $S_n= \operatorname{conv}(A^n_0,A^n_1,\dots,A^n_d)$, $n=1,2,\dots$, be an infinite
  sequence of simplices. If the sequence violates the condition from
  Definition~\ref{def-krizek-general} then it also violates the generalized maximum angle
  condition in Definition~\ref{def-semiregular-family}.
  \label{semi_implies_krizek}
\end{theorem}

\begin{proof}
By Lemma~\ref{lemma-subsequence}, we can assume that the limiting vectors $\vec{t}_{XY}$ for all pairs
  $\{X_n,Y_n\}\subset\{A_0^n,\dots,A_d^n\}$ exist.
Let $S_n'\subseteq S_n$ be as in Definition~\ref{def-krizek-general}, and assume that there
is an infinite sequence of dihedral angles of $S_n'$
tending to $\pi$.
We need to show that all the $\sin_d$ tend to zero.
Let $d'$ be the dimension of $S_n'$.
By reordering, we can assume that the vertices of $S_n'$ are $A^n_0,\dots,A^n_{d'}$.
Then since one of the dihedral angles tends to $\pi$, the set of limit vectors $\vec{t}_{XY}$
for all pairs $\{X_n,Y_n\}\subset \{A_0^n,\dots,A_{d'}^n\}$ only span a space of dimension $\leq d'-1$.
Adding the vectors (say) $\vec{t}_{A_0A_{d'+1}},\vec{t}_{A_0A_{d'+2}},\dots,\vec{t}_{A_0A_d}$ adds no more than $d-d'$ to the dimension of the span.
Any remaining vector $\vec{t}_{XY}$ will be in this span, more precisely in the span of $\vec{t}_{A_0X}$ and $\vec{t}_{A_0Y}$.
Therefore any choice of $d$ vectors from the set of $\vec{t}_{XY}$ can only span a space of dimension $\leq d-1$, so $\sin_d$ is zero by Lemma~\ref{lemma-cont} b).
By continuity of $\sin_d$ (Lemma~\ref{lemma-cont} c)), the generalized maximum angle condition in Definition~\ref{def-semiregular-family} is violated.
\end{proof}

\begin{theorem}
  The conditions in Definition~\ref{def-krizek-general} and in Definition~\ref{def-semiregular-family} are equivalent.
  \label{theorem-main-equivalence}
\end{theorem}

\begin{proof}
  Use the contradiction argument and Theorem~\ref{semi_implies_krizek} we observe  that
  the condition of Definition~\ref{def-semiregular-family} implies the condition in
  Definition~\ref{def-krizek-general}.

To prove the statement of the theorem in the opposite direction, we generalize the construction
proposed by M. K\v r\'{\i}\v zek in \cite{Kri-1992}. (See also Remark \ref{rem_comparison}
following this proof for the precise relationship of two constructions.)
Assume that the condition in Definition~\ref{def-krizek-general} holds.
We show that the generalized maximum angle
  condition in Definition~\ref{def-semiregular-family} also holds.
The proof is by induction on the dimension $d$, where the base case is known ($d=2, 3$).
Let $S\in \mathcal{T}_h\in\mathcal{F}$ be a $d$-dimensional simplex, and let the bound
on the angles be $\gamma_0$ as in Definition~\ref{def-krizek-general}.
By induction, for any subsimplex $S'\subset S$ of dimension $d-1$ with vertex set
contained in the vertex set of $S$, one can choose $d-1$ unit vectors along the
edges so that the $\sin_{d-1}$ applied to these vectors is bounded from below by
some constant.
When restricted to the set of $d$ unit vectors, both $\sin_d$ and $\operatorname{meas}_d$ of the parallelotope spanned by the vectors are continuous functions which are zero precisely on the compact subset of linearly dependent sets of vectors. Therefore the existence of a positive lower bound for the values of one of them on a family implies the existence of a positive lower bound for the other.
Let  $C>0$ be a lower bound for the measure of these parallelotopes.
Order the vertices of $S$ as $A_0\dots A_d$ so that the largest dihedral angle
$\alpha_{d-1,d}$ is between the facets opposite to $A_{d-1}$ and $A_d$.
By induction, choose $d-1$ unit vectors $\vec{t}_1,\dots,\vec{t}_{d-1}$ along
the edges of the subsimplex $A_0\dots A_{d-1}$ so that  $\operatorname{meas}_{d-1}(\vec{t}_1\dots \vec{t}_{d-1})>C$.
Now look at the subsimplex $A_0\dots A_{d-2}A_d$ opposite to $A_{d-1}$.
In this subsimplex, we can also choose $d-1$ unit vectors $\vec{u}_1\dots \vec{u}_{d-1}$ with the same property.
At least one of these must be along an edge incident to $A_d$; define this vector (or one of these vectors) to be $\vec{t}_d$.

Let $y$ be the height of $\vec{t}_d$ over the subspace spanned by
$\{\vec{u}_1\dots \vec{u}_{d-1}\}\setminus \{\vec{t}_d\}$.
Then $y$ is the height of a parallelotope with bounded volume, and we have:
\begin{equation*}
  C<\operatorname{meas}_{d-1}(\vec{u}_1\dots \vec{u}_{d-1}) =
  y\cdot \operatorname{meas}_{d-2}\big(\{\vec{u}_1\dots \vec{u}_{d-1}\}\setminus \{\vec{t}_d\}\big)\leq y\cdot 1.
\end{equation*}
The last inequality holds because all the involved vectors are unit vectors.
Let $z$ be the height of $\vec{t}_d$ over $\{ \vec{t}_1\dots \vec{t}_{d-1} \}$.
Finally,
\begin{equation*}
\begin{array}{lll}
\operatorname{meas}_d(\vec{t}_1\dots \vec{t}_d)&=&z \operatorname{meas}_{d-1}(\vec{t}_1\dots \vec{t}_{d-1})\\
&=&y\sin \alpha_{d-1,d}\operatorname{meas}_{d-1}(\vec{t}_1\dots \vec{t}_{d-1}). \end{array}
\end{equation*}
Now $y\geq C$ and $\operatorname{meas}_{d-1}(\vec{t}_1\dots \vec{t}_{d-1})\geq C$.
Since $\alpha_{d-1,d}$ is the largest dihedral angle, it lies in the interval $(\gamma_1,\gamma_0)$,
where $\gamma_1$ is the dihedral angle in the regular $d$-simplex (see e.g. \cite{Mae-2013}).
Therefore $\sin \alpha_{d-1}\geq \min \{\sin \gamma_1,\sin \gamma_0\}$.
Thus the measure of the spanned parallelotope is bounded from below:
\begin{equation*}
\operatorname{meas}_d(\vec{t}_1\dots \vec{t}_d)\geq C^2 \min \{\sin \gamma_1,\sin \gamma_0\}.
\end{equation*}
We conclude that $\sin_d$ is also bounded from below by \eqref{d-sine}.
\end{proof}
\begin{remark}
  M. K\v r\'{\i}\v zek uses a similar construction in \cite[pp. 517--518]{Kri-1992} to find three
  vectors in a tetrahedron satisfying the same conditions.
  In his proof, he starts with an arbitrary triangle and a large dihedral angle incident to a chosen
  edge of this triangle.
Two unit vectors are chosen along edges in this triangle.
He then chooses a vector based on angles in the other triangle used to compute the dihedral angle.
His choice of two vectors in the arbitrary triangle is similar to our choice of $\vec{t}_1, \dots, \vec{t}_{d-1}$.
His procedure using angles to choose the third vector is equivalent to our choice of $\vec{t}_d$.
Our bounds are slightly easier because we avoid the choice of an arbitrary subsimplex/triangle,
so we will only need bounds for the largest, not the second largest, dihedral angle in a given dimension.
\label{rem_comparison}
\end{remark}

\subsection*{Jamet's definition}
In \cite{Jam}, Jamet estimates interpolation error in terms of an angle $\theta$ defined as follows (with notation adapted to the present article).
Let $E=\{\vec{e}_i\}_{i=1}^d$ be a set of unit vectors in $\mathbb{R}^d$. For any other unit vector $\vec{u}\in S^{d-1}$, define $\theta_i(\vec{u})$ to be the angle between $\vec{u}$ and the line through $\vec{e}_i$.
Then define $\theta$ by
\begin{equation}
\theta = \max_{\vec{u}\in S^{d-1}}\min_{i=1,\dots,d}\theta_i(\vec{u}).
\label{eq-Jamet-condition}
\end{equation}
Jamet obtains formulas bounding interpolation errors where a factor $1/\cos\theta$ appears.
In particular, in Exemple 1 in \cite{Jam}, $E$ is chosen to be a set of unit vectors along the edges of a simplex as in Definition~\ref{def-semiregular-family}.

\begin{definition}
    A family $\mathcal F  = \{\mathcal T_h \}_{h \to 0}$ of face-to-face partitions of a polytope
  $\overline \Omega \subset {\bf R}^d$ into
$d$-simplices is said to satisfy  {\it Jamet's condition} if
    there exists a constant $\theta_0<\pi/2$ such that for all simplices $S\in \mathcal{T}_h\in\mathcal{F}$, one can choose $d$ unit vectors along edges of $S$ in such a way that the $\theta$ computed in Equation~\eqref{eq-Jamet-condition} satisfies $\theta\leq\theta_0$.
\label{def-jamet}
\end{definition}

\begin{theorem}
Jamet's condition in Defintion~\ref{def-jamet} is equivalent to the condition in Definition~\ref{def-semiregular-family}, and consequently also to the condition in Definition~\ref{def-krizek-general}.
\end{theorem}
\begin{proof}
If the maximum in Equation~\eqref{eq-Jamet-condition} is $\pi/2$, it is clear that the set $E$ of $d$ vectors is in fact not linearly independent, and that the maximum is attained for any vector which is a normal vector to a $(d-1)$-dimensional subspace that contains all the vectors.
This was observed also by Jamet.
The existence of a bound $\theta_0<\pi/2$ means that the set $E$ is separated from the subset of $\big(S^{d-1}\big)^{d}$ consisting of linearly dependent sets of vectors.
As seen in the proof of Theorem~\ref{theorem-main-equivalence}, this is equivalent to a lower bound for $\sin_d$ applied to the same set of vectors.
This concludes the proof.
\end{proof}

\begin{remark}
A priori, checking a finite number of angles is an easier task than finding the maximum angle $\theta$.
To prove the existence of a bound in one of the conditions given the existence of a bound for the other is simpler than giving formulas.
For $d=2$, this is easy: If $\gamma_0<\pi$ is an upper bound for the angles of the triangles, $\theta_0=\gamma/2$ is an upper bound for the angles in Jamet's condition.
In \cite{Rand}, Rand proves the case $d=3$ of the above theorem by explicitly computing $\theta_0$ given $\gamma_0$ and vice versa.
\end{remark}
%

%
%
%

\ifx\undefined\bysame
\newcommand{\bysame}{\leavevmode\hbox to3em{\hrulefill}\,}
\fi

\end{document}